\documentclass[reqno, 11pt]{amsart}
\usepackage[letterpaper, margin=1in]{geometry}
\usepackage{amsthm, amsmath, amsfonts, amssymb, needspace}
\title{Stein-Weiss, and power weight Korn type Hardy-Sobolev Inequalities in \(L^1\) norm}
\author{Wen Qi Zhang}
\address{Mathematical Sciences Institute\\ Australian National University\\Acton ACT 2601, Australia}
\email{WenQi.Zhang@anu.edu.au}
\usepackage{mathdots}
\usepackage{microtype}
\usepackage{xcolor}
\usepackage{verbatim}
\usepackage{enumitem}
\usepackage{thmtools}
\usepackage{tikz}
\usepackage{needspace}
\usetikzlibrary{shadings}

\usetikzlibrary{arrows.meta}

\usepackage{pgf}
\usepackage{pgfplots}
\usepgflibrary {shadings} 
\usepgfplotslibrary{fillbetween}
\usetikzlibrary{intersections}

\newlength\tindent
\theoremstyle{plain}
\newtheorem{thm}{Theorem}[section]
\newtheorem{prop}[thm]{Proposition}
\newtheorem{lemma}[thm]{Lemma}
\newtheorem{corollary}[thm]{Corollary}
\newtheorem{question}[thm]{Question}

\newtheoremstyle{boldremark}
{\dimexpr\topsep/2\relax} 
{\dimexpr\topsep/2\relax} 
{}          
{}          
{\bfseries} 
{.}         
{.5em}      
{}          

\theoremstyle{boldremark}
\newtheorem{remark}[thm]{Remark}
\newtheorem{eg}[thm]{Example}

\theoremstyle{definition}
\newtheorem{defn}[thm]{Definition}

\usepackage{graphicx}
\usepackage{hyperref}
\DeclareGraphicsExtensions{.eps}

\usepackage{fancyhdr}
\usepackage{mathtools}
\usepackage{mathrsfs}

\DeclarePairedDelimiter\floor{\lfloor}{\rfloor}

\def\N{{\mathbb N}}

\def\R{{\mathbb R}}

\def\ep{\varepsilon}
\def\c00{c_{00}}

\DeclarePairedDelimiter\norm{\lVert}{\rVert}

\begin{document}
	\maketitle
	\vspace{-0.35cm}
\begin{abstract}
	We extend the \(L^1\) Stein-Weiss inequalities studied by De N\'{a}poli and Picon \cite{NapPic} in two ways: First we address an open question posed by the authors about whether the cocanceling condition was necessary for some of their Stein-Weiss inequalities. We replace the cocanceling condition with a weaker vanishing moment assumption, and under this assumption extend the \(L^1\) Stein-Weiss inequalities to \(L^1\big(|x|^{a } dx\big)\) data for all positive, non-integer exponents \(a\). Second, in relation to integer exponents, while \cite{NapPic} showed that Stein-Weiss fails for \(L^1\big(|x| dx\big)\) data, we prove a weaker Korn type Hardy-Sobolev inequality. These inequalities were previously inaccessible due to the growth of \(|x|\), and we demonstrate a specific example on \(\R^2\) of where the original duality estimate by Bousquet and Van Schaftingen \cite{can} for canceling operators can be improved.
	\end{abstract}
\section{Introduction}
This work is motivated by a question posed by De N\'{a}poli and Picon in \cite{NapPic} concerning endpoint \(L^1\) Stein-Weiss inequalities. The classical \(L^p\) versions of these inequalities describe the mapping properties of the Riesz Potential on power weighted \(L^p\) spaces (see e.g. \cite{sw58}). Our convention for the Riesz Potential will be (for \(n\geq 2\) and \(0<k<n\)) to denote it by (a constant multiple depending on \(n, \ k\) of)
\begin{align*}
	\mathcal{I}_{k}f(x)&:=\int_{\R^n}\frac{f(y)}{|x-y|^{n-k}}dy, \qquad f\in\mathcal{S}(\R^n).
\end{align*}
This extends to vector fields valued in a finite dimensional real vector space \(E\), and we define \(\mathcal{I}_k\) acting on these fields by identifying \(E\) with a standard Euclidean space and integrating componentwise. It is known that endpoint \(L^1\) Stein-Weiss inequalities are more delicate than their \(L^p\) counterparts, and certain extra assumptions must be placed on \(f\) to obtain the desired inequality. In \cite{NapPic} the authors showed that a differential condition on \(f\), such as divergence-free, and more generally cocanceling condition \ref{cocancelcond}, was sufficient to prove Stein-Weiss inequalities for vector fields \(f\in L^1(\R^n;E,|x|^a \ dx)\) when \(0\leq a<1\). This exponent range is larger than expected, with previous results only addressing \(a<0\). Excluding the case \(a=0\), we are able to relax the differential condition condition \(\textup{div }f=0\) (and its cocanceling generalisation) to simply a mean zero assumption \(\int_{\R^n}f=0\). This assumption covers a broader class of functions and vector fields. We actually prove a slightly more general statement for positive, non-integer exponents under a vanishing moment assumption:	
\begin{thm}\label{newSW}\(\)\\
	Let \(E\) be a finite dimensional vector space, \(n\geq 2\), \(0<k<n\) and suppose that \(0<a\), \(q\in[1,\frac{n}{n-k})\) and \(b\) are the exponents satisfying \(\frac{n+b}{q}=n-k+a\). If \(a\notin \N\) then the inequality
	\begin{equation}
		\left(\int_{\R^n}|x|^b|\mathcal{I}_{k}f(x)|^qdx\right)^{\frac{1}{q}}\lesssim \int_{\R^n}|x|^a|f(x)|dx \label{SW}
	\end{equation}
	holds for all \(f\in C^{\infty}_c(\R^n,E)\) (with implied constant independent of \(f\)) satisfying the vanishing moment assumption: for all multi-indices \(\gamma\leq \floor{a}\), \(\int_{\R^n} x^{\gamma}f(x) \ dx=0\). Here \(\floor{a}\) denotes the floor of \(a\).
\end{thm}
When restricted to the range \(0<a<1\), this theorem answers in the negative a question of De N\'{a}poli and Picon in \cite{NapPic} about the necessity of the cocanceling condition. It is known from \cite{cocan} (Theorem 1.4) that the cocanceling condition on \(f\) will imply \(\int_{\R^n} f =0\), and so this assumption covers a broader range of functions and vector fields. Hence, we must exclude the case \(a=0\) where it was shown by De N\'{a}poli and Picon that the cocanceling condition is necessary and sufficient for Stein-Weiss inequalities to hold. On the other hand, our result also highlights that it is possible to extend beyond \(a\geq 1\), and the remaining case of integer \(a\) is of particular interest to us.\\
\\
When \(a\geq 1\) the vanishing moment assumption is no longer implied by the cocanceling condition, and we are curious about the case \(a=1\), which is excluded by our non-integer assumption on \(a\). It was previously shown in \cite{NapPic} that inequality \eqref{SW} fails under the cocanceling condition, but we wonder if there exists a (stronger) condition which is sufficient to imply Stein-Weiss. We could only partially answer this problem in a weaker context. The setting we have in mind is the symmetric derivative, symmetrized gradient, or Korn operator on \(\R^2\), which we define for a vector field \(u\in C^1(\R^2,\R^2)\) by
\begin{align*}
	\textup{D}_{\textup{sym}}u&=\begin{bmatrix}
		\partial_{1}u_1 & \frac{\partial_{1}u_2+\partial_2u_1}{2}\\
		\frac{\partial_2u_1+\partial_{1}u_2}{2} & \partial_{2}u_2
	\end{bmatrix}\\
	&=\frac{1}{2}\left(Du+(Du)^T\right).
\end{align*}
If we set \(f=\textup{D}_{\textup{sym}}u\) for a compactly supported, smooth vector field in Theorem \ref{newSW}, and use the heuristic idea
\begin{align*}
	|u(x)|&\lesssim |\mathcal{I}_1[\textup{D}_{\textup{sym}} u](x)|
\end{align*} 
then we observe that Theorem \ref{newSW} implies power weight Hardy-Sobolev inequalities for the exponent range \(0<a<1\). However this range can be improved to include \(1\leq a < 2\) (and \(q=2\)):
\begin{restatable}{thm}{kornsobthm}\label{kornsob}\(\)\\
	Suppose that \(q\in[1,2]\), \(1\leq a<2\), and \(b\) are the exponents satisfying \(\frac{2+b}{q}=1+a\). Then
	\begin{align}
		\left(\int_{\R^2}|x|^{b}|u(x)|^q dx\right)^{\frac{1}{q}}&\lesssim \int_{\R^2}|x|^a|\textup{D}_{\textup{sym}}u(x)|dx \label{kornsobineq}
	\end{align}
	for all \(u\in C^{\infty}_c(\R^2,\R^2)\) with implied constant independent of \(u\).
\end{restatable}
This theorem is derived via an improved duality estimate using an additional cancellation of \(\textup{D}_{\textup{sym}}\), which we could not generalise to the \textit{canceling} framework of Van Schaftingen \cite{cocan}. This is because our additional cancellation appears to require certain conditions which we could not verify in general, even in the basic case of \(\nabla\) on \(\R^3\). However, we expect some argument along this line to work, and we discuss this briefly in Section \ref{future_work}.
\subsection{Notations} Throughout this work we will use Vinogradov notation, that is we will write \(A\lesssim B\) to mean there exists a constant \(C>0\) such that \(A\leq CB\). The implied constant may depend on the parameters, and we will state what the implied constant is independent of in our theorems and lemmas.
\subsection{Outline} The rest of this work is organised as follows. In the next section, we highlight the relevant background concerning \(L^1\) Stein-Weiss and Hardy-Sobolev inequalities including the definitions for the \textit{canceling} framework. We also recall the relevant theorem of \cite{NapPic} in our notation. In Section \ref{techredsec} we present the routine and well known steps/reductions of Theorems \ref{newSW} and \ref{kornsob}. Then in Section \ref{SWsubsection} we demonstrate how the vanishing moment assumption can be used to close the argument of Theorem \ref{newSW}. In Section \ref{kornsobsec} we present and explain our improved cancellation which allows us to close the argument in Theorem \ref{kornsob}. Finally, in Section \ref{future_work} we discuss a possible extension for our key lemmas and show via examples barriers to obtaining Theorem \ref{kornsob} for all elliptic and canceling operators.
\section{Preliminaries}
\subsection{Classical Results}
Hardy-Sobolev inequalities and their generalisations are fundamental tools in analysis and PDE. They provide a priori estimates which are the basic tools for studying the existence, regularity and uniqueness of solutions. When \(1<p<\infty\) these inequalities easily generalise to other differential operators via  Calder\'{o}n and Zygmund theory:
\begin{thm}[A.P. Calder\'{o}n and A. Zygmund, 1952 \cite{calderonzygmund}]\(\)\\
	Let \(V, \ E\) be finite dimensional vector spaces and \(1<p<\infty\). Suppose that \(A(D)\) is a constant coefficient, injectively elliptic operator of order \(k\) on \(\R^n\) from \(V\) to \(E\). The estimate
	\begin{align*}
		\norm{D^ku}_{L^p}&\leq C\norm{A(D)u}_{L^p}
	\end{align*}
	holds for every \(u\in C^{\infty}_c(\R^n,V)\) with implied constant independent of \(u\).
\end{thm}
Unfortunately, the above theorem fails for \(p=1\) (and \(p=\infty\)). The \(p=1\) failure is captured by Ornstein's famous non-inequality \cite{Ornstein}, and the explicit vector valued statement is covered by Kirchheim and Kristensen in \cite{KK}, which states that only trivial \(L^1\) inequalities can hold between homogeneous, linear differential operators (see \cite{KK} Theorem 1.3):
\begin{thm}[D. Ornstein, 1962 \cite{Ornstein}]\(\)\\
	Let \(V,\ W_1,\ W_2\) be finite dimensional inner product spaces and consider two constant coefficient differential operators \(A_i(D)\), \(i=\{1,2\}\) on \(\R^n\) from \(V\) to \(W_i\) homogeneous of order \(k\). Then
	\begin{align*}
		\norm{A_1(D)u}_{L^1}\leq c\norm{A_2(D)u}_{L^1}
	\end{align*}
	holds for all \(u\in C^{\infty}_c(\R^n,V)\) if and only if there exists a \(T\in\mathcal{L}(W_2,W_1)\) such that
	\begin{align*}
		A_1(D)&=T\circ A_2(D).
	\end{align*}
\end{thm} 
A similar picture is also present for Stein-Weiss inequalities, which fail at \(p=1\) without further assumptions:
\begin{thm}[E. M. Stein and G. Weiss, 1958 \cite{sw58}]\label{Stein_Weiss_Ineq}\(\)\\
	Let \(n\geq 1\), \(0<k<n\) and suppose that \(1<p<\frac{n}{k}\), \(a<n(p-1)\), \(b>-n\), \(q\in[p,\frac{np}{n-kp}]\) satisfy \(\frac{n+b}{q}=\frac{n+a-kp}{p}\). Then the inequality
	\begin{equation*}
		\left(\int_{\R^n}|x|^b|\mathcal{I}_{k}f(x)|^qdx\right)^{\frac{1}{q}}\lesssim \left(\int_{\R^n}|x|^a|f(x)|^pdx\right)^{\frac{1}{p}}
	\end{equation*}
	holds for all \(f\in C^{\infty}_c(\R^n)\) with implied constant independent of \(f\)
\end{thm}
The requirement that \(p<\frac{n}{k}\) is not necessary, and we introduce it because when \(p\geq \frac{n}{k}\) the relationship between the permissible parameters \(a, b, q\) changes slightly, and so assuming \(p<\frac{n}{k}\) simplifies the presentation. We briefly mention that the \(p\geq \frac{n}{k}\) setting is more closely related to Morrey's inequality and the Hardy-Sobolev inequalities correspond to \(p<\frac{n}{k}\). Our focus is on relaxing the \(1<p\) endpoint, so assuming \(p<\frac{n}{k}\) places no additional constraints. Previously, it was known that the requirement \(1<p\) can be relaxed to \(1\leq p\) if we assume \(a<0\) and the strict inequality \(q<\frac{np}{n-kp}\) (see e.g. \cite{NapPic} or \cite{DAncona_Luca}).
\subsection{An Adjacent Result} We mention in passing that progress on the endpoint \(p=1, \ q=\frac{n}{n-k}\) is possible when \(a<0\) via another argument by D'Ancona and Luc\`{a} \cite{DAncona_Luca}. The authors showed that the desired inequality holds under the (different) hypothesis that \(f\) has Fourier support in an annulus. 
\subsection{Canceling Framework}
Unfortunately, the above results are limited and unsatisfying when \(p=1\) and \(a\geq 0\). An important question studied by Bousquet, Van Schaftingen \cite{can} and Van Schaftingen \cite{cocan} is if weaker inequalities can hold, for instance \(L^1\) multidimensional Hardy-Sobolev inequalities. In these works, the authors introduced a general \textit{canceling} framework which captures the necessary and sufficient conditions for Hardy-Sobolev inequalities of elliptic differential operators:
\begin{defn}[Canceling]\label{defcancel}\(\)\\
	Given finite dimensional vector spaces \(V,E\) we say that \(A(D)\) is a (constant coefficient) canceling differential operator (on \(\R^n\) from \(V\) to \(E\)) if one has
	\begin{align*}
		\bigcap_{\xi\in\R^n\setminus\{0\}}\text{Image}[A(\xi)]=\{0_E\}.
	\end{align*}
\end{defn}
This condition guarantees that (when \(A(D)\) is canceling) there exists a cocanceling operator \(L(D)\) which annihilates it:
\begin{defn}[Cocanceling]\label{defcocancel}\(\)\\
	Given finite dimensional vector spaces \(E,F\) we say that \(L(D)\) is a (constant coefficient) cocanceling differential operator (on \(\R^n\) from \(E\) to \(F\)) if one has
	\begin{align*}
		\bigcap_{\xi\in\R^n\setminus\{0\}}\ker[L(\xi)]=\{0_E\}.
	\end{align*}
\end{defn} 
Cocanceling is important in this context as it allows us to guarantee the existence of a polynomial \(P(x)\) such that \(L(D)P(x)=Id\). The key idea here is that when we have a pair \(A(D)\) and \(L(D)\) which are canceling and cocanceling respectively, then we can insert \(L(D)\) via \(L(D)P(x)=Id\) into integral expressions involving \(A(D)\), and after an integration by parts trick, we can access \(L(D)A(D)=0\) which allows us to obtain better decay than expected in \(|\mathcal{I}_kA(D)u|\). 
\subsection{Cocanceling Condition}\label{cocancelcond}
The argument described above also works if we replace \(A(D)u\) by \(f\) and impose the weaker hypothesis (on \(f\)) that \(L(D)\) is cocanceling and \(L(D)f=0\). We refer to this hypothesis as the cocanceling condition, and this leads to the following result of \cite{NapPic} which extends Stein-Weiss inequalities to \(0\leq a<1\). From this viewpoint, the cocanceling condition appears natural:
\begin{thm}[\cite{NapPic} Theorem 1.2]\label{starttheorem}\(\)\\
Let \(E, \ F\) be finite dimensional vector spaces, \(n\geq 2\), \(0<k<n\) and suppose that \(0\leq a< 1\), \(b\in\R\), and \(q\in[1,\frac{n}{n-k})\) satisfy \(\frac{n+b}{q}=n-k+a\). Let \(L(D)\) be a differential operator on \(\R^n\) from \(E\) to \(F\).  If \(L(D)\) is cocanceling the inequality
\begin{equation}
	\left(\int_{\R^n}|x|^b|\mathcal{I}_{k}f(x)|^qdx\right)^{\frac{1}{q}}\lesssim \int_{\R^n}|x|^a|f(x)|dx \label{SWineq2}
\end{equation}
holds for all \(f\in C^{\infty}_c(\R^n,E)\) satisfying \(L(D)f=0\) (in the sense of distributions) with implied constant independent of \(f\). Conversely, if all non-zero \(f\) satisfying \(L(D)f=0\) satisfy the inequality \eqref{SWineq2} for \(a=0\), then \(L(D)\) must be cocanceling.
\end{thm}
However, as we discussed in the introduction, the cocanceling condition is not necessary when \(0<a<1\), and can be replaced with a mean zero assumption. This is due to an observation that when \(a\notin \N\), one can use the local integrability of \(|x|^{\ep-n}\) when \(\ep>0\) to relax the cancellation requirement to the vanishing moment assumption. We elaborate on this in Section \ref{SWsubsection}.
\subsection{Further Cancellations} The case \(a=1\) (and more generally \(a\in \N\))  requires stronger cancellations. The methods of Section \ref{SWsubsection} fail because the expression \(|x|^{-n}\) appears in place of \(|x|^{\ep-n}\), and this is no longer locally integrable. Hence further analysis is required to refine the region of integration to avoid the origin. This scenario already occurs when \(a=0\) and can be handled with the cocanceling condition. However, when \(a=1\) there are additional terms which cannot be treated in the same way as the \(a=0\) case, and further cancellations are needed. We highlight some form of additional cancellations are necessary, as De N\'{a}poli and Picon already showed in \cite{NapPic} that \eqref{SWineq2} fails under the cocanceling condition. We observed that such cancellations are possible in a weaker scenario involving \(\textup{D}_{\textup{sym}}\) and its Green's function on \(\R^2\). See Section \ref{kornsobsec} for further details.\\
\\
These further cancellations are also why Theorem \ref{kornsob} does not generalise to the same setting as the previous work (\cite{can}, \cite{NapPic}, \cite{HouniePicon}, \cite{cocan}). However, we are hopeful that some argument along these lines works and we explore some possibilities and limitations in Section \ref{future_work}.



\subsection{Exponent Diagrams}
While not necessary for understanding the proofs that follow, it is convenient to visualise the the relationship between \(a, b, q\) in Theorem \ref{starttheorem}. For a given \(n\geq 2\) and \(k<n\), the (possibly false) inequality \eqref{SWineq2} can be represented by a point \((x,y)\) with the relationship \(a=x\) and \(b=y\). Under this identification, \(q=\frac{b+n}{a+(n-k)}\), is the slope of the line joining \((a,b)\) and \((k-n,-n)\), and hence the restriction \(q\in[1,\frac{n}{n-k})\) limits the permissible slopes in our diagram. For completeness we also include the region \(k-n<a<0\) implied by Theorem \ref{Stein_Weiss_Ineq} when \(p=1\). The inequalities which hold are described by the triangular region between the lines \(b=\frac{n}{n-k} \ a\), \(b=a-k\), and \(a<1\):
\begin{figure}[!h]
	\centering
	\begin{tikzpicture}[scale = 1]
		\draw[thick,<->] (-1.2,0) -- (1.2,0) node[anchor=north west] {\(a\)};
		\draw[thick,<->] (0,-2.2) -- (0,2.2) node[anchor=south east] {\(b\)};
		\coordinate (x1) at (-1,-2);
		\coordinate (Q1) at (1,0);
		\coordinate (Q2) at (1,2);
		\draw (-1 cm,-3pt) -- (-1 cm,3pt) node[anchor=north, outer sep = 3pt] {$k-n$};
		\draw (-3pt,-2 cm) -- (3pt, -2cm) node[anchor=west, inner xsep = 3pt] {$-n$};
		\draw[thick, red] (-0.9,-1.9) -- (1,0) node[anchor=south west] {\(b=a-k\)};
		\draw[thick, dashed, blue] (-1,-2) -- (1,2) node[anchor=west, inner xsep = 4pt, outer ysep = 18pt] {\(b=\frac{n}{n-k} \ a\)};
		\shade[nearly transparent, top color = blue, bottom color = red, shading angle = 55]
		(x1) -- (Q1) -- (Q2) -- (x1) -- cycle;
	\end{tikzpicture}
	\caption{Permissible exponents under the cocanceling condition.}
	\label{td}
\end{figure}
\newpage
In view of Theorem \ref{newSW}, if sufficiently many moments of \(f\) are zero, then Figure \ref{td} (for vanishing moment \(f\)) can be updated to include the following region (note the removal of the vertical lines corresponding to \(a\in \N\)):
\begin{figure}[!h]
	\centering
	\begin{tikzpicture}[scale=0.75]
		\draw[thick,<->] (-1.2,0) -- (1.2,0) node[anchor=north west] {\(a\)};
		\draw[thick,<->] (0,-2.2) -- (0,2.2) node[anchor=south east] {\(b\)};
		\coordinate (x1) at (1,0);
		\coordinate (Q1) at (3.95,2.95);
		\coordinate (Q2) at (3.95,7.9);
		\coordinate (x2) at (1,2);
		\draw (-1 cm,-3pt) -- (-1 cm,3pt) node[anchor=north, outer sep = 3pt] {$k-n$};
		\draw (-3pt,-2 cm) -- (3pt, -2cm) node[anchor=west, inner xsep = 3pt] {$-n$};
		\draw[thick, -{Stealth[scale=1.33]}, red] (-1,-2) -- (4,3) node[anchor=south west] {\(b=a-k\)};
		\draw[thick, dashed, -{Stealth[scale=1.33]}, blue] (-1,-2) -- (4,8) node[anchor=west, inner xsep = 4pt, outer ysep = 18pt] {\(b=\frac{n}{n-k} \ a\)};
		\shade[nearly transparent, top color=blue,
		bottom color =red, shading angle = 53]
		(x1) -- (Q1) -- (Q2) -- (x2) -- cycle;
		\draw[ultra thick, -, white] (1,0) -- (1,2);
		\draw[ultra thick, -, white] (2,0) -- (2,4);
		\draw[ultra thick, -, white] (3,0) -- (3,6);
	\end{tikzpicture}
	\caption{Permissible exponents under the vanishing moment assumption.}
\end{figure}\(\)\\
Finally, in view of Theorem \ref{kornsob}, Figure \ref{td} interpreted for power weight Hardy-Sobolev inequalties with \(f=\textup{D}_{\textup{sym}}u\) on \(\R^2\) can be updated to include the following endpoints:
\begin{figure}[!h]
	\centering
	\begin{tikzpicture}[scale=1]
		\draw[thick,<->] (-1.2,0) -- (2.2,0) node[anchor=north west] {\(a\)};
		\draw[thick,<->] (0,-2.2) -- (0,4) node[anchor=south east] {\(b\)};
		\coordinate (x1) at (1,0);
		\coordinate (Q1) at (2,1);
		\coordinate (Q2) at (2,4);
		\coordinate (x2) at (1,2);
		\foreach \x in {-1,1}
		\draw (\x cm,1pt) -- (\x cm,-1pt) node[anchor=north] {$\x$};
		\foreach \y in {-2,0,2}
		\draw (1pt,\y cm) -- (-1pt,\y cm) node[anchor=east] {$\y$};
		\draw[thick, red] (-1,-2) -- (2,1) node[anchor=south west] {\(b=a-1\)};
		\draw[thick, blue] (-1,-2) -- (2,4) node[anchor=west, inner xsep = 4pt, outer ysep = 18pt] {\(b=2a\)};
		\shade[nearly transparent, top color =blue,
		bottom color =red, shading angle = 56]
		(x1) -- (Q1) -- (Q2) -- (x2) -- cycle;
		\draw[ultra thick, dashed, white] (2,0) -- (2,4);
	\end{tikzpicture}
	\caption{Permissible exponents for power weight Hardy-Sobolev inequalties for \(\textup{D}_{\textup{sym}}u\) on \(\R^2\)}
\end{figure}
\subsection{A Trivial Inequality}
We close this section by highlighting that Hardy-Sobolev inequalities at \(q=1\) and \(a=1\) coincide with certain Poincar\'{e} inequalities (see e.g. \cite{poinckorn} for Poincar\'{e}-Korn). A trivial one dimensional integration by parts proof exists for this result:
\begin{restatable}{lemma}{IBPlemma}\label{IBPlemma}\(\)\\
	For \(u\in C^{\infty}_c(\R^n)\) one has
	\begin{align*}
		\int_{\R^n}|u(x)|dx&\leq\int_{\R^n}|x||\partial_{x_1}u(x)|dx.
	\end{align*}
\end{restatable}
We emphasize that this lemma would also work for the derivative on \(\R\), and it can be checked by using \(1=\partial_{x_1}(x_1)\), and approximating \(|u(x)|\) by 
\begin{align*}
	u_{\ep}(x)=\left(u(x)^2+\ep\right)^{\frac{1}{2}}-\ep^{\frac{1}{2}}.
\end{align*}
\section{Technical Reductions}\label{techredsec}
To simplify the presentation in the proof of Theorems \ref{newSW} and \ref{kornsob} we remove the `easy' part and present it as a separate lemma. For each fixed \(y\in\R^n\), the function \(|x-y|^{q(k-n)}|x|^b\) is already locally integrable in some regions, and can be estimated directly here. On the complement of this region, one can show that the leading term of the Taylor series expansion for \(|x-y|^{q(k-n)}\) also satisfies the desired estimates. We treat these two terms via the lemma below, and the lower order terms will be the focus of later sections.
\begin{lemma}\label{technical_reduction}\(\)\\
	Let \(E,\ V\) be finite dimensional vector spaces, \(n\geq 2\), \(0<k<n\) and suppose that \(a\geq 0\), \(b\in \R\), and \(q\in[1,\frac{n}{n-k})\) satisfy \(\frac{n+b}{q}=n-k+a\). Let \(m=\floor{a}\) denote the floor of \(a\). If \(\rho\in L^{\infty}(\R)\) is identically \(1\) on \(\left[0,\frac{1}{2\lambda}\right]\) and compactly supported in \(\left[0,\frac{1}{\lambda}\right]\) for some \(\lambda>1\), and if \(G\in C^{m+1}(\R^n\setminus\{0\},\mathcal{L}(E,V))\) is positively homogeneous of degree \(k-n\), define 
	\begin{align*}
		H(x,y)&:=\rho\left(\frac{|y|}{|x|}\right)\sum_{|\alpha|\leq m}y^{\alpha}\frac{\partial_x^{\alpha}G(x)}{\alpha!}.
	\end{align*}
	Then the inequality
	\begin{align*}
		\left(\int_{\R^n}|x|^b \left|\int_{\R^n}G(x-y)f(y)dy\right|^qdx\right)^{\frac{1}{q}}&\lesssim \int_{\R^n}|x|^a|f(x)|dx+\left(\int_{\R^n}|x|^{b}\left|\int_{\R^n}H(x,y)f(y)dy\right|^q dx\right)^{\frac{1}{q}}
	\end{align*}
	holds for all \(f\in C^{\infty}_c(\R^n, E)\) with implied constant independent of \(f\).
\end{lemma}
\begin{remark}
	In application, we do not need the full generality, we set \(\lambda=2\) and only have two cases of \(\rho\) in mind: In Theorem \ref{newSW} we simply use \(\rho=\chi_{[0,\frac{1}{4}]}\). In Lemma \ref{key} we assume \(\rho\) is a smooth function, identically \(1\) on \([0,\frac{1}{4}]\) and compactly supported in \([0,\frac{1}{2}]\).
\end{remark}
\begin{proof}\(\)\\
	\textbf{Step 1 (Indicator Function):}\\
	First consider \(\rho=\chi_{\left[0,\frac{1}{2\lambda}\right]}\) for \(\lambda>1\). Write the decomposition
	\begin{align*}
		H(x,y)&=\rho\left(\frac{|y|}{|x|}\right)\sum_{|\alpha|\leq m}y^{\alpha}\frac{\partial_x^{\alpha}G(x)}{\alpha!}\\
		K(x,y)&=G(x-y)-\rho\left(\frac{|y|}{|x|}\right)\sum_{|\alpha|\leq m}y^{\alpha}\frac{\partial_x^{\alpha}G(x)}{\alpha!}
	\end{align*}
	here when \(m=0\) we mean \(\sum_{|\alpha|\leq m}y^{\alpha}\frac{\partial_x^{\alpha}G(x)}{\alpha!}=G(x)\). By triangle inequality for weighted \(L^q\),
	\begin{align*}
		\left(\int_{\R^n}|x|^b\left|G\ast f(x)\right|^q dx\right)^{\frac{1}{q}}&\leq \left(\int_{\R^n}|x|^{b}\left|\int_{\R^n}H(x,y)f(y)dy\right|^q dx\right)^{\frac{1}{q}}+\left(\int_{\R^n}|x|^{b}\left|\int_{\R^n}K(x,y)f(y)dy\right|^qdx\right)^{\frac{1}{q}}
	\end{align*}
	and hence it suffices to demonstrate
	\begin{align*}
		\left(\int_{\R^n}|x|^{b}\left|\int_{\R^n}K(x,y)f(y)dy\right|^qdx\right)^{\frac{1}{q}}&\lesssim \int_{\R^n}|y|^a|f(y)|dy \tag{\(I_2\)} \label{I_2}.
	\end{align*}
	We split the left hand side into two terms, on the support of \(\rho\) and off the support of \(\rho\):
	\begin{align*}
		\left(\int_{\R^n}|x|^{b}\left|\int_{2\lambda|y|<|x|}K(x,y)f(y)dy\right|^qdx\right)^{\frac{1}{q}}+\left(\int_{\R^n}|x|^{b}\left|\int_{2\lambda|y|\geq |x|}K(x,y)f(y)dy\right|^qdx\right)^{\frac{1}{q}}.
	\end{align*}
	\textbf{Step 1.a (On The Support Of \(\rho\)):}\\
	When \(2\lambda|y|<|x|\)  we have 		
	\begin{align}
		|K(x,y)|&\lesssim |y|^{m+1}|x|^{k-n-m-1} \label{9}
	\end{align}
	This follows by Taylor expanding \(G(x-y)\). It remains to check the remainder term from Taylor expanding \(G\). The naive estimate for this term is \(|y|^{m+2}|z|^{k-n-m-2}\), for some \(z\) on the line segment between \(x\) and \(x-y\). This can be improved by recalling that in the region \(2\lambda|y|<|x|\), any point on the line segment must satisfy both \(|y|<|z|\) and \(\frac{(2\lambda-1)|x|}{2\lambda}\leq|z|\). This would then verify \eqref{9}, and in combination with Minkowski's inequality we obtain
	\begin{align*}
		\left(\int_{\R^n}|x|^{b}\left|\int_{2\lambda|y|<|x|}K(x,y)f(y)dy\right|^qdx\right)^{\frac{1}{q}}&\lesssim \int_{\R^n}|f(y)||y|^{m+1}\left(\int_{|x|>2\lambda|y|}|x|^{q(k-n-m-1)+b}dx\right)^{\frac{1}{q}}dy.
	\end{align*}
	The integral in \(x\) on the right hand side is controlled by a constant multiple of \(|y|^{a-m-1}\) because \(|x|^{q(k-n-m-1)+b} \) is integrable on \(|x|>2\lambda|y|\) (note the numerology, substitute \(b\), and use \(a-m-1<0\)). In other words the previous display can be simplified to
	\begin{align*}
		\int_{\R^n}|f(y)||y|^{m+1}\left(\int_{|x|>2\lambda|y|}|x|^{q(k-n-m-1)+b}dx\right)^{\frac{1}{q}}dy&\lesssim \int_{\R^n}|f(y)||y|^ady
	\end{align*}
	which is the desired conclusion.\\
	\textbf{Step 1.b (Away From The Support Of \(\rho\)):}\(\)\\
	When \(2\lambda|y|\geq |x|\) we use
	\begin{align*}
		|K(x,y)|&\lesssim |x-y|^{k-n}
	\end{align*}
	and in combination with Minkowski gives
	\begin{align*}
		\left(\int_{\R^n}|x|^{b}\left|\int_{2\lambda|y|\geq |x|}K(x,y)f(y)dy\right|^qdx\right)^{\frac{1}{q}}&\lesssim \int_{\R^n}|f(y)|\left(\int_{2\lambda|y|\geq |x|}|x-y|^{q(k-n)}|x|^{b}dx\right)^{\frac{1}{q}}dy.
	\end{align*}
	The integral in \(x\) can be evaluated by decomposing \(B(0,2\lambda|y|)\) into \(B\left(y,\frac{|y|}{2\lambda}\right)\) and \(B(0,2\lambda|y|)\cap B\left(y,\frac{|y|}{2\lambda}\right)^c\). On \(B\left(y,\frac{|y|}{2\lambda}\right)\) the kernel \(|x-y|^{q(k-n)}\) is integrable because \(q\in [1,\frac{n}{n-k})\), and on \(B(0,2\lambda|y|)\cap B\left(y,\frac{|y|}{2\lambda}\right)^c\) one has \(|x|^b\) is integrable because \(b>-n\) (recall the conditions \(a\geq 0\) and \(q\in [1,\frac{n}{n-k})\) imply \(b>-k\)). Hence
	\begin{align*}
		\left(\int_{2\lambda|y|\geq |x|}|x-y|^{q(k-n)}|x|^{b}dx\right)^{\frac{1}{q}}&\lesssim |y|^{\frac{n+b}{q}+k-n}
	\end{align*}
	and the numerology gives \(\frac{n+b}{q}+k-n=a\), so \eqref{I_2} follows.\\
	\textbf{Step 2 (General Case):}\\
	The case where \(\rho\in L^{\infty}(\R)\) with \(\rho=1\) on \(\left[0,\frac{1}{2\lambda}\right]\) and \(\text{supp }\rho\subset \left[0,\frac{1}{\lambda}\right]\) remains. The same argument as above works with minor modifications and a similar argument was already described in \cite{can}, \cite{NapPic}, \cite{HouniePicon}. However, for completeness we include the modifications needed: Recall \eqref{I_2}, we instead write the LHS as
	\begin{align*}
		\left(\int_{\R^n}|x|^{b}\left|\int_{\lambda|y|<|x|}K(x,y)f(y)dy\right|^qdx\right)^{\frac{1}{q}}{}&+{}\left(\int_{\R^n}|x|^{b}\left|\int_{\lambda|y|\geq|x|}K(x,y)f(y)dy\right|^qdx\right)^{\frac{1}{q}}.
	\end{align*}
	We claim that \eqref{9} from Step 1.a above still follows on \(\lambda|y|<|x|\) as before: When \(2\lambda|y|<|x|\) the same Taylor expansion step works. The modification is on \(\lambda|y|<|x|\leq 2\lambda |y|\), where one can use \(|x|\approx|y|\) to arrive at the same conclusion. Following the same steps as before (Minkowksi's inequality, and substituting \(b\)) we obtain:
	\begin{align*}
		\left(\int_{\R^n}|x|^{b}\left|\int_{\lambda|y|<|x|}K(x,y)f(y)dy\right|^qdx\right)^{\frac{1}{q}}&\lesssim
		\int_{\R^n}|f(y)||y|^{m+1}\left(\int_{\lambda|y|<|x|}|x|^{b}|x|^{q(k-n-m-1)}\right)^{\frac{1}{q}}dy\\
		&\lesssim \int_{\R^n}|f(y)||y|^{a}dy.
	\end{align*}
	For the other term (Step 1.b), we still use
	\begin{align*}
		|K(x,y)|&\lesssim |x-y|^{k-n}
	\end{align*}
	but on the set \(\lambda|y|\geq |x|\) (instead of \(2\lambda|y|\geq |x|\)). The decomposition changes to accommodate this, we decompose \(B\left(0,\lambda|y|\right)\) (rather than \(B(0,2\lambda|y|)\)) into \(B\left(y,\frac{|y|}{2\lambda}\right)\) and \(B\left(0,\lambda|y|\right)\cap B\left(y,\frac{|y|}{2\lambda}\right)^c\). Following the same local integrability reasoning we arrive at
	\begin{align*}
		\left(\int_{\R^n}|x|^{b}\left|\int_{\lambda|y|\geq|x|}K(x,y)f(y)dy\right|^qdx\right)^{\frac{1}{q}}&\lesssim\int_{\R^n}|f(y)||y|^ady
	\end{align*}
	and so \eqref{I_2} follows in this case too.
\end{proof}
\section{Stein-Weiss Inequalities for Zero-Moment Functions}\label{SWsubsection}
In view of Lemma \ref{technical_reduction}, to close the proof of Theorem \ref{newSW} we need to justify why the lower order Taylor series terms satisfy the desired bounds. Let \(\rho=\chi_{[0,\frac{1}{4}]}\), our problem essentially reduces to justifying why the following function of \(r>0\)
\begin{align*}
	J(r)&:=\int_{\R^n}f(y)\rho\left(\frac{|y|}{r}\right)dy
\end{align*}
decays quickly enough so that \(J(|x|)|x|^{\ep-n}\) becomes integrable in \(x\) on \(\R^n\) when \(\ep>0\). The key observation is that if \(f(y)\) is a mean zero function, then one has the identity
\begin{align*}
	J(r)=-\int_{\R^n}f(y)\left(1-\rho\left(\frac{|y|}{r}\right)\right)dy
\end{align*}
which can be derived by subtracting off the mean. From this identity the desired decay is immediate, and this is our strategy in the following proof:

\begin{proof}[Proof of Theorem \ref{newSW}]\(\)\\
	To simplify notation we set \(\rho=\chi_{[0,\frac{1}{4}]}\). Let \(m:=\floor{a}\) denote the floor of \(a\) and write
	\begin{align*}
		H(x,y)&=\rho\left(\frac{|y|}{|x|}\right)\sum_{|\alpha|\leq m}y^{\alpha}\frac{\partial_x^{\alpha}(|x|^{k-n})}{\alpha!}
	\end{align*}
	here when \(m=0\) we mean \(\sum_{|\alpha|\leq m}y^{\alpha}\frac{\partial_x^{\alpha}(|x|^{k-n})}{\alpha!}=|x|^{k-n}\). By Lemma \ref{technical_reduction} (with \(G(x)=|x|^{k-n}\)) it suffices to check
	\begin{align}
		\left(\int_{\R^n}|x|^{b}\left|\int_{\R^n}H(x,y)f(y)dy\right|^q dx\right)^{\frac{1}{q}}&\lesssim \int_{\R^n}|y|^a|f(y)|dy \label{Hbound}.
	\end{align}
	To simplify notation we write
	\begin{align*}
		J(|x|)&:= \int_{\R^n}H(x,y)f(y)dy.
	\end{align*}
	The vanishing moment assumption, which was \(\int_{\R^n}y^{\gamma}f(y) \ dy=0\) for all \(\gamma\leq \floor{a}=m\), implies
	\begin{align*}
		\int_{\R^n}\sum_{|\alpha|\leq m}y^{\alpha}\frac{\partial_x^{\alpha}(|x|^{k-n})}{\alpha!}f(y)dy&=0.
	\end{align*}
	The term on the left can be subtracted from \(J(|x|)\) to give the identity
	\begin{align*}
		|J(|x|)|&=\left|\int_{\R^n}\left(1-\rho\left(\frac{|y|}{|x|}\right)\right)\sum_{|\alpha|\leq m}y^{\alpha}\frac{\partial_x^{\alpha}(|x|^{k-n})}{\alpha!}f(y)dy\right|.
	\end{align*}
	The support of \(1-\rho\left(\frac{|y|}{|x|}\right)\) is the region \(4|y|>|x|\), and combining the above identity for \(J(|x|)\) with Minkowski's integral inequality (and triangle inequality) gives,
	\begin{align*}
		\left(\int_{\R^n}|x|^{b}|J(|x|)|^{q}dx\right)^{\frac{1}{q}}&\lesssim\sum_{\ell\leq m} \int_{\R^n}|f(y)||y|^{\ell}\left(\int_{\{y\in\R^n:4|y|>|x|\}}|x|^{b+q(k-n-\ell)}dx\right)^{\frac{1}{q}}dy.
	\end{align*}
	The integral in \(x\) on the RHS is bounded by a constant multiple of \(|y|^{a-\ell}\), to see this use the assumptions to replace \(b\) in the exponent of \(x\) and realise that \(|x|^{q(a-\ell)-n}\) is locally integrable because \(a>\ell\). This implies \eqref{Hbound}, completing the argument. 
\end{proof}
\section{Korn Type Hardy-Sobolev Inequalities on \(\R^2\)}\label{kornsobsec}
In the previous section we used that \(|x|^{\ep-n}\) was integrable around \(0\) when \(\ep>0\). However, when \(a\in\N\) a term in our Taylor series behaves like \(|x|^{-n}\) which is not integrable around \(0\). In previous work for the \(a=0\) case (\cite{can}, \cite{NapPic}, \cite{HouniePicon}) this was treated by Lemma \ref{original_can}, which refines the region of integration using the canceling assumption. Note that our Lemma \ref{key} is an improvement in a special case.
\begin{lemma}[\cite{can}, Lemma 2.2]\label{original_can}\(\)\\
	Let \(A(D)\) be a linear differential operator of order \(k\) on \(\R^n\) from \(V\) to \(E\). If \(A(D)\)  is injectively elliptic and canceling, then there exists \(C\in\R\) and \(\ell\in\N\setminus\{0\}\) such that for every \(u\in C^{\infty}_c(\R^n,V)\) and every \(\varphi\in C^\ell(\R^n\setminus\{0\},E)\) that satisfies for every \(j\in\{0,\dots, \ell\}\), \(|x|^j|D^j\varphi|\in L^1_{\textup{loc}}(\R^n)\),
	\begin{align*}
		\left|\int_{\R^n}\varphi\cdot A(D)u\right|&\leq C\sum_{j=1}^{\ell}\int_{\R^n}|A(D)u(x)||x|^j|D^j\varphi(x)|dx.
	\end{align*}
\end{lemma}
The special case of the above lemma we are interested in is for \(A(D)=\textup{D}_{\textup{sym}}{}\) and for a specific, radial \(\varphi\) which we denote by \(\rho\). In this case, one has \(\ell=2\) (see Lemma \ref{key}) and the precise statement is:
\newpage
\begin{corollary}\label{original_key}\(\)\\
	Given \(\rho\in C^{\infty}_c([0,\frac{1}{2}])\) with \(\rho=1\) on \([0,\frac{1}{4}]\), then for each \(x\in\R^2\setminus\{0\}\), the inequality
	\begin{align*}
		\left|\int_{\R^2}\textup{D}_{\textup{sym}} u(y)\rho\left(\frac{|y|}{|x|}\right)dy\right|&\lesssim \int_{|y|\approx|x|}|\textup{D}_{\textup{sym}}u(y)|dy
	\end{align*}
	holds for all \(u\in C^{\infty}_c(\R^2,\R^2)\), with implied constant independent of \(u\). Here \(|y|\approx|x|\) denotes the region \(\left\{y{}\in{}\R^2 {} : {} \frac{1}{4} {}\leq{}\frac{|y|}{|x|} {}\leq{} \frac{1}{2}\right\}\).
\end{corollary}
Note that on the left hand side, \(\rho\left(\frac{|y|}{|x|}\right)\) is supported when \(|y|\leq \frac{|x|}{2}\), whereas on the right hand side we exclude the region \(|y|<\frac{|x|}{4}\). As discussed at the start of this section, this is critical to show the main results of \cite{can}.\\
\\
To see why Lemma \ref{original_can} implies the corollary, note that on the support of \(\left|\rho'\left(\frac{|y|}{|x|}\right)\right|\) and \(\left|\rho''\left(\frac{|y|}{|x|}\right)\right|\) one has \(\frac{1}{4}\leq \frac{|y|}{|x|}\leq \frac{1}{2}\). This also simplifies the powers of \(\frac{|y|}{|x|}\) which appear when applying the lemma.
\subsection{Key Improvement}\label{keyimprov}\(\)\\
We improve Corollary \ref{original_key} for the special case \(A(D)=\textup{D}_{\textup{sym}}\) on \(\R^2\). For the rest of this section, we define \(A(D)\) acting on a vector field \(u\in C^{\infty}_c(\R^2,\R^2)\) by 
\begin{align*}
	A(D)u&:=\begin{bmatrix}
		\partial_1 u_1\\
		\partial_2u_1+\partial_1u_2\\
		\partial_2u_2
	\end{bmatrix}.
\end{align*}
To deduce a Green's function for \(A(D)\) we notice the relationship
\begin{align*}
	\begin{bmatrix}
		\partial_{1} & \partial_{2} & -\partial_{1}\\
		-\partial_{2} & \partial_{1} & \partial_{2}
	\end{bmatrix}A(D)u&=\begin{bmatrix}
		\Delta u_1 \\
		\Delta u_2
	\end{bmatrix}.
\end{align*}
Hence a Green's function is represented by (a scalar multiple of)
\begin{align*}
	G(x):&=\begin{bmatrix}
		\partial_{x_1} & \partial_{x_2} & -\partial_{x_1}\\
		-\partial_{x_2} & \partial_{x_1} & \partial_{x_2}
	\end{bmatrix}\log|x|\\
	&=\begin{bmatrix}
		x_1 & x_2 & -{x_1}\\
		-{x_2} & {x_1} & {x_2}
	\end{bmatrix}\frac{1}{|x|^2}.
\end{align*}
Note that \(G(x)\) is differentiable for all \(x\in\R^2\setminus \{0\}\) and positively homogeneous of degree \(-1\), and \(\partial_{x_i}G(x)\) is positively homogeneous of degree \(-2\) for every \(1\leq i\leq n\). With this in mind, our improvement is:
\begin{lemma}\label{key}\(\)\\
	Given \(\rho\in C^{\infty}_c([0,\frac{1}{2}])\) with \(\rho=1\) on \([0,\frac{1}{4}]\), then for each \(x\in\R^2\setminus\{0\}\), the inequality
	\begin{align*}
		\left|\int_{\R^2}y\cdot DG(x)A(D) u(y)\rho\left(\frac{|y|}{|x|}\right)dy\right|&\lesssim \frac{1}{|x|}\int_{|y|\approx|x|}|A(D)u(y)|dy
	\end{align*}
	holds for all \(u\in C^{\infty}_c(\R^2,\R^2)\), with implied constant independent of \(u\). Here \(|y|\approx|x|\) denotes the region \(\left\{y{}\in{}\R^2 {} : {} \frac{1}{4} {}\leq{}\frac{|y|}{|x|} {}\leq{} \frac{1}{2}\right\}\), and \(y\cdot DG(x)=\sum_{i=1}^ny_i\partial_{x_i}G(x)\).
\end{lemma}
The improvement is (again) the region of integration appearing on the right hand side, as on the support of \(\rho\left(\frac{|y|}{|x|}\right)\), the term \(|y\cdot DG(x)|\lesssim |y||x|^{-2}\) appearing on the LHS is controlled by a constant multiple of \(|x|^{-1}\) which is what appears on the RHS.\\ 
\\
Note that Lemma \ref{original_can} does not imply this result, as we would want to set (for each fixed \(x\in\R^2\setminus\{0\}\))
\begin{align*}
	\varphi_x(y)&:=y\cdot DG(x)\rho\left(\frac{|y|}{|x|}\right).
\end{align*} 
Applying the product rule (in \(y\)) to compute \(D\varphi_{x}\) results in terms which do not have derivatives of \(\rho\), which allow us to exclude the region \(|y|<\frac{|x|}{4}\). As mentioned at the start of this section, this exclusion was essential in proving Theorem \ref{kornsob}. To rectify this, we found an improvement in the proof of Lemma \ref{original_can} from \cite{can} by combining the kernel of \(L(D)\) and identities involving the Green's function. These steps lead to an additional cancellation, which closed the argument.
\begin{proof}[Proof of Lemma \ref{key}]\(\)\\
	We first reduce to the case where \(|x|=1\), that is if we set \(\omega=\frac{x}{|x|}\) it suffices to prove
	\begin{align}
		\left|\int_{\R^2}y\cdot DG(\omega)A(D)u(y)\rho(|y|)dy\right|&\lesssim \int_{|y|\approx 1}|A(D)u(y)|dy. \label{goal}
	\end{align}
	Here \(|y|\approx1\) denotes the region \(\left\{y{}\in{}\R^2 {} : {} \frac{1}{4} {}\leq{}|y| {}\leq{} \frac{1}{2}\right\}\). Indeed, fix \(x\in\R^2\setminus\{0\}\) and write \(x=r\omega\) for some \(r>0\) and unit vector \(\omega\). If we assume \eqref{goal} holds then we test against the function \(u_{r}(y)=u(r y)\) giving
	\begin{align*}
		\left|\int_{\R^2}y\cdot DG(\omega)A(D)u(r y)\rho(|y|)dy\right|&\lesssim \int_{|y|\approx 1}|A(D)u(r y)|dy.
	\end{align*}
	The desired conclusion follows by changing variable \(\tilde{y}=\frac{y}{r}\) and using the positive homogeneity of \(DG\).\\
	\\
	Returning to the proof of \eqref{goal}, as \(A(D)\) was canceling, we find the following choice of cocanceling \(L(D)\) such that \(L(D)A(D)u=0\) appropriate:
	\begin{align*}
		L(D)&:=\begin{bmatrix}
			1 & 0 & 0
		\end{bmatrix}\partial_{y_2}^2 + \begin{bmatrix}
			0 & -1 & 0
		\end{bmatrix}\partial_{y_1}\partial_{y_2}+\begin{bmatrix}
			0 & 0 & 1
		\end{bmatrix}\partial_{y_1}^2\\
		&=L_{(0,2)}\partial_{y_2y_2}+L_{(1,1)}\partial_{y_1y_2}+L_{(2,0)}\partial_{y_1y_1}.
	\end{align*}
	Define \(K(y)\) by
	\begin{align*}
		K(y)&=\begin{bmatrix}
			\frac{y_2^2}{2}\\
			-y_1y_2\\
			\frac{y_1^2}{2}
		\end{bmatrix}.
	\end{align*}
	and note that
	\begin{align*}
		\sum_{|\alpha|=2}\partial_y^{\alpha}K(y)L_{\alpha}=Id_{\R^3}.
	\end{align*}
	This identity can be inserted into the integrand on the LHS of \eqref{goal} in the following manner,
	\begin{align*}
		y\cdot DG(\omega)&=\sum_{|\alpha|=2}y\cdot DG(\omega)\partial_{y}^{\alpha}K(y)L_{\alpha}.
	\end{align*}
	Our goal is to integrate by parts to recover \(\partial^{\alpha}_yL_{\alpha}\) terms as they will cancel when applied to \(A(D)u\). In order to achieve this, we use the product rule to rearrange the RHS of the above so that all other functions of \(y\) can be written under \(\partial^{\alpha}_y\). The first step is to write
	\begin{align*}
		\sum_{|\alpha|=2}y\cdot DG(\omega)\partial_{y}^{\alpha}K(y)L_{\alpha}&=\sum_{|\alpha|=2}\Bigg(\partial^{\alpha}_y(y\cdot DG(\omega)K(y))L_{\alpha}-\sum_{e_i\leq \alpha}{\alpha \choose e_i}\partial_{x_i}G(\omega)\partial^{\alpha-e_i}_yK(y)L_{\alpha}\Bigg).
	\end{align*}
	With this identity in mind our desired estimate would be implied by (note the support properties of \(\rho'\) and \(\rho''\)):
	\begin{align}
		\begin{split}
			\left|\int_{\R^2}\sum_{|\alpha|=2}\partial^{\alpha}_y(y\cdot DG(\omega)K(y))L_{\alpha}A(D)u(y)\rho\left(|y|\right)dy\right|\lesssim{}&\\
			\int_{\R^2}|A(D)u(y)|&\left(\left|\rho'\left(|y|\right)\right|+\left|\rho''\left(|y|\right)\right|\right)dy
		\end{split} \label{K_term}\\
		\begin{split}
			\left|\int_{\R^2}\sum_{|\alpha|=2}\sum_{e_i\leq \alpha}{\alpha \choose e_i}\partial_{x_i}G(\omega)\partial^{\alpha-e_i}_{y}K(y)L_{\alpha}A(D)u(y)\rho\left(|y|\right)dy\right|\lesssim{}&\\
			\int_{\R^2}|A(D)u(y)|& \left(\left|\rho'\left(|y|\right)\right|+\left|\rho''\left(|y|\right)\right|\right)dy
		\end{split}\label{P_term}
	\end{align}
	We consider \eqref{K_term} first. Recall that our goal was to integrate by parts to recover \(\partial_y^{\alpha}L_{\alpha}\) terms and so we want to rewrite the integrand which appears on the left hand side by moving the \(\rho\left(|y|\right)\) term under \(\partial_y^{\alpha}\). In a similar process to before, a rearrangement of the product rule gives the following identity:
	\begin{align*}
		\begin{split}
			\sum_{|\alpha|=2}\partial^{\alpha}_y\Big( y\cdot DG(\omega)K(y)\Big) L_{\alpha}A(D)u(y)\rho\left(|y|\right)&={}\sum_{|\alpha|=2}\Bigg(\partial^{\alpha}_y \Big( y\cdot DG(\omega)K(y)\rho\left(|y|\right)\Big)L_{\alpha}A(D)u(y)\Bigg.\\
			&\ \Bigg.{}-{}\sum_{\substack{\beta\leq \alpha\\\beta\neq \alpha}}\partial^{\beta}_y\Big(y \cdot DG(\omega)K(y)\Big) L_{\alpha}A(D)u(y)\partial_{y}^{\alpha-\beta}\left[\rho\left(|y|\right)\right]\Bigg).
		\end{split}
	\end{align*}
	The first term on the right hand side cancels when integrated in \(y\), as integrating by parts and using \(L(D)A(D)u=0\) gives
	\begin{align}
		\int_{\R^2}\sum_{|\alpha|=2}\partial^{\alpha}_y\Big(y\cdot DG(\omega)K(y)\rho\left(|y|\right)\Big)L_{\alpha}A(D)u(y)dy&=0. \label{cancel}
	\end{align}
	Hence integrating both sides of the product rule identity would yield the better identity
	\begin{align*}
		\begin{split}
			\int_{\R^2}\sum_{|\alpha|=2}\partial^{\alpha}_y\Big( y\cdot DG(\omega)K(y)\Big) L_{\alpha}A(D)u(y)&\rho\left(|y|\right)dy=\\
			{}-{}\int_{\R^2}\sum_{|\alpha|=2}&\Bigg(\sum_{\substack{\beta\leq \alpha\\ \beta\neq \alpha}}\partial^{\beta}_y\Big(y \cdot DG(\omega)K(y)\Big) L_{\alpha}A(D)u(y)\partial_{y}^{\alpha-\beta}\left[\rho\left(|y|\right)\right]\Bigg)dy.
		\end{split}
	\end{align*}
	Applying absolute values to both sides, and using triangle inequality and replacing \(|y|\) terms by a constant on the right hand side leads to
	\begin{align*}
		\begin{split}
			\left|\int_{\R^2}\sum_{|\alpha|=2}\partial^{\alpha}_y\Big( y\cdot DG(\omega)K(y)\Big) L_{\alpha}A(D)u(y)\rho\left(|y|\right)dy\right|&\lesssim{}\\
			\int_{\R^2}&|A(D)u(y)| \left(\left|\rho'\left(|y|\right)\right|+\left|\rho''\left(|y|\right)\right|\right)dy.
		\end{split}
	\end{align*}
	This verifies \eqref{K_term}. To verify \eqref{P_term}, recall that in \eqref{K_term} we used an integration by parts step \eqref{cancel}, which worked because we could move \(\partial^{\alpha}_y\) off of a term which was independent of \(\alpha\) and hence apply the canceling relation \(L(D)A(D)u=0\). As written, the corresponding term in \eqref{P_term} would be
	\begin{align*}
		\sum_{e_i\leq \alpha}{\alpha \choose e_i}\partial_{x_i}G(\omega)\partial^{\alpha-e_i}_{y}K(y).
	\end{align*}
	If we find \(K_i\) such that \(\partial_{y_i}K_i(y)=K(y)\) to write
	\begin{align}
		\sum_{e_i\leq \alpha}{\alpha \choose e_i}\partial_{x_i}G(\omega)\partial^{\alpha}_{y}K_i(y) \label{adep}
	\end{align}
	the integration by parts cancellation in step \eqref{cancel} would fail without further analysis as the above term is not independent of \(\alpha\). However, if we view \(\partial^{\alpha}_y\) as a linear operator on homogeneous polynomials of degree \(3\) we notice that \(\ker \partial_y^{\alpha}\) is large. We claim that \eqref{adep} can be made independent of \(\alpha\) by choosing
	\begin{align*}
		K_1(y)&:=\begin{bmatrix}
			\frac{y_2^2y_1}{2}\\
			\frac{y_2^3}{6}-\frac{y_1^2y_2}{2}\\
			\frac{y_1^3}{6}
		\end{bmatrix}\\
		K_2(y)&:=\begin{bmatrix}
			\frac{y_2^3}{6}\\
			\frac{y_1^3}{6}-\frac{y_2^2y_1}{2}\\
			\frac{y_1^2y_2}{2}
		\end{bmatrix}.
	\end{align*}
	Note that \(\partial_{y_i}K_i(y)=K(y)\), so we just need to verify the following \(\alpha\) independence: We claim there exists compatible maps \(T(\omega), \ P(y)\) (independent of \(\alpha\)) such that for every \(|\alpha|=2\),
	\begin{align*}
		\eqref{adep}=\partial^{\alpha}_y(T(\omega)P(y)).
	\end{align*} 
	Setting
	\begin{align}
		T(\omega)P(y)
		&:=2\partial_{x_1}G(\omega)K_1(y) \label{magic}
	\end{align}
	verifies the claim because we have the surprising identity \(\partial_{x_1}G(\omega)K_1(y)=\partial_{x_2}G(\omega)K_2(y)\), which drastically simplifies the sum over \(e_i\leq \alpha\) appearing in \eqref{adep}. This identity can be checked by calculating:
	\begin{align*}
		\partial_{x_1}G(\omega)K_1(y)&=\begin{bmatrix}
			\left(\frac{y_2^2y_1}{2}-\frac{y_1^3}{6}\right)(\omega_2^2-\omega_1^2)+2\left(\frac{y_2^3}{6}-\frac{y_1^2y_2}{2}\right)\omega_1\omega_2\\
			2\left(-\frac{y_2^2y_1}{2}+\frac{y_1^3}{6}\right)\omega_1\omega_2+\left(\frac{y_2^3}{6}-\frac{y_1^2y_2}{2}\right)(\omega_2^2-\omega_1^2)
		\end{bmatrix}
	\end{align*}
	and 
	\begin{align*}
		\partial_{x_2}G(x)K_2(y)&=\begin{bmatrix}
			2\left(\frac{y_2^3}{6}-\frac{y_1^2y_2}{2}\right)\omega_1\omega_2+\left(\frac{y_1^3}{6}-\frac{y_2^2y_1}{2}\right)(\omega_1^2-\omega_2^2)\\
			\left(-\frac{y_2^3}{6}+\frac{y_1^2y_2}{2}\right)(\omega_1^2-\omega_2^2)+2\left(\frac{y_1^3}{6}-\frac{y_2^2y_1}{2}\right)\omega_1\omega_2\\
		\end{bmatrix}.
	\end{align*}
	Thus we may apply the same argument as \eqref{K_term}, deducing (in place of \eqref{cancel}) that
	\begin{align*}
		\int_{\R^2}\sum_{|\alpha|=2}\partial^{\alpha}_y\Big(2\partial_{x_1}G(\omega)K_1(y)\rho\left(|y|\right)\Big)L_{\alpha}A(D)u(y)dy&=0
	\end{align*} 
	and hence by the same product rule manipulations that follow we conclude \eqref{P_term}, which concludes the proof of \eqref{goal}.
\end{proof}
\subsection{Proof of Theorem \ref{kornsob}}\label{proof}\(\)\\
We now have the tools to prove our main result, we first prove the result for \(q\in [1,2)\):
\begin{proof}[Proof of Theorem \ref{kornsob} for \(q\in [1,2)\).]\(\)\\
	Fix a function \(\rho\in C^{\infty}_c(\R)\) with \(\rho=1\) on \([0,\frac{1}{4}]\) and \(\text{supp }\rho\subset [0,\frac{1}{2}]\) with bounded (by a large constant) first and second derivatives. Recall the representation \(A(D)\) and corresponding Green's function \(G\) chosen for \(\textup{D}_{\textup{sym}}\) in Lemma \ref{key}. Set
	\begin{align*}
		H(x,y)&=\rho\left(\frac{|y|}{|x|}\right)\left(G(x)+y\cdot DG(x)\right)
	\end{align*}
	and so by Lemma \ref{technical_reduction} it suffices to show
	\begin{align*}
		\left(\int_{\R^2}|x|^{b}\left|\int_{\R^2}\rho\left(\frac{|y|}{|x|}\right)G(x)A(D)u(y)dy\right|^{q}dx\right)^{\frac{1}{q}}&\lesssim \int_{\R^2}|y|^a|A(D)u(y)|dy \tag{\(I_{1,1}\)}\label{I11}\\
		\left(\int_{\R^2}|x|^{b}\left|\int_{\R^2}\rho\left(\frac{|y|}{|x|}\right)y\cdot DG(x)A(D)u(y)dy\right|^{q}dx\right)^{\frac{1}{q}}&\lesssim \int_{\R^2}|y|^a|A(D)u(y)|dy \tag{\(I_{1,2}\)}\label{I12}.
	\end{align*}
	For \eqref{I11}, using the homogeneity of \(G(x)\) one has
	\begin{align*}
		\Bigg(\int_{\R^2}|x|^{b}\left|\int_{\R^2}\rho\left(\frac{|y|}{|x|}\right)G(x)A(D)u(y)dy\right|&^{q}dx\Bigg)^{\frac{1}{q}}&\lesssim \Bigg(\int_{\R^2}|x|^{b-q}\left|\int_{\R^2}\rho\left(\frac{|y|}{|x|}\right)A(D)u(y)dy\right|^{q}&dx\Bigg)^{\frac{1}{q}}
	\end{align*}
	and applying Corollary \ref{original_key} gives
	\begin{align*}
		\begin{split}
			\Bigg(\int_{\R^2}|x|^{b}\left|\int_{\R^2}\rho\left(\frac{|y|}{|x|}\right)G(x)A(D)u(y)dy\right|^{q}&dx\Bigg)^{\frac{1}{q}}\lesssim{}\\
			& \Bigg(\int_{\R^2}|x|^{b-q}\left|\int_{|y|\approx|x|}|A(D)u(y)|dy\right|^{q}dx\Bigg)^{\frac{1}{q}}.
		\end{split}
	\end{align*}
	For \eqref{I12}, applying Lemma \ref{key} we again obtain
	\begin{align*}
		\begin{split}
			\Bigg(\int_{\R^2}|x|^{b}\left|\int_{\R^2}\rho\left(\frac{|y|}{|x|}\right)y\cdot DG(x)A(D)u(y)dy\right|^{q}&dx\Bigg)^{\frac{1}{q}}\lesssim{}\\
			& \Bigg(\int_{\R^2}|x|^{b-q}\left|\int_{|y|\approx|x|}|A(D)u(y)|dy\right|^{q}dx\Bigg)^{\frac{1}{q}}.
		\end{split}
	\end{align*}
	Recall that \(|y|\approx|x|\) means the region \(\left\{y{}\in{}\R^2 {} : {} \frac{1}{4} {}\leq{}\frac{|y|}{|x|} {}\leq{} \frac{1}{2}\right\}\), it suffices to show
	\begin{align*}
		\left(\int_{\R^2}|x|^{b-q}\left|\int_{|y|\approx |x|}|A(D)u(y)|dy\right|^{q}dx\right)^{\frac{1}{q}}&\lesssim \int_{\R^2}|y|^a|A(D)u(y)|dy \tag{\(\tilde{I}_1\)} \label{I1} .
	\end{align*}
	Applying Minkowski to the left hand side, one arrives at
	\begin{align*}
		\int_{\R^2}|A(D)u(y)|\left(\int_{|y|\approx|x|}|x|^{b-q} dx\right)^{\frac{1}{q}}dy
	\end{align*}
	and noting that
	\begin{align*}
		\left(\int_{\frac{|x|}{2}\leq|y|\leq|x|}|x|^{b-q}dx\right)^{\frac{1}{q}}&\lesssim |y|^{\frac{b-q+2}{q}}\\
		&=|y|^a
	\end{align*}
	where the last line follows since \(\frac{2+b}{q}=1+a\), we conclude \eqref{I1} completing the proof.\\
\end{proof}
It remains to prove Theorem \ref{kornsob} when \(q=2\):
\begin{proof}[Proof of Theorem \ref{kornsob} for \(q=2\).]\(\)\\
	From \cite{cocan}, one may conclude that for all \(u\in C^{\infty}_c(\R^2,\R^2)\)
	\begin{align*}
		\left(\int_{\R^2}|u(x)|^2dx\right)^{\frac{1}{2}}&\lesssim\int_{\R^2}|\textup{D}_{\textup{sym}}u(x)|dx
	\end{align*}
	with implied constant independent of \(u\). Let \(1\leq a <2\) and apply the above to the function \(|x|^au(x)\). After an application of the product rule we conclude
	\begin{align*}
		\left(\int_{\R^2}|x|^{2a}|u(x)|^2dx\right)^{\frac{1}{2}}&
		\lesssim\int_{\R^2}|x|^a|\textup{D}_{\textup{sym}} u(x)|dx+\int_{\R^2}|x|^{a-1}|u(x)|dx.
	\end{align*}
	Use Theorem \ref{kornsob} when \(q=1\) to bound the last term on the RHS:
	\begin{align*}
		\int_{\R^2}|x|^{a-1}|u(x)|dx&\lesssim \int_{\R^2}|x|^a|\textup{D}_{\textup{sym}}u(x)|.
	\end{align*}
\end{proof}
\section{Future Work}\label{future_work}
It is not hard to see how the proof of Lemma \eqref{key} can be modified to prove the following:
\begin{prop}\label{keyprop}\(\)\\
	Suppose that \(A(D)\) is a constant coefficient differential operator on \(\R^n\) from \(V\) to \(E\) of order \(k\). Suppose also that \(A(D)\) is elliptic and canceling. If \(A(D)\) admits (a specific representation for)
	\begin{enumerate}[label=(A\arabic*)]
		\item a Green's function, \(G(x)\), \label{G}
		\item a cocanceling operator \(L(D)\) of order \(\ell\) on \(\R^n\) from \(E\) to \(F\) satisfying \(L(D)A(D)=0\), \label{L}
		\item and a collection of linear mappings \(\{K_{\alpha}\}_{|\alpha|=\ell}\) from \(F\) to \(E\) satisfying \(\sum_{|\alpha|=\ell}K_{\alpha}L_{\alpha}=Id_{E}\) \label{K}
	\end{enumerate}
	such that there exists a finite dimensional vector space \(M\), and maps \(T(x)\), \(P(y)\) (independent of \(|\alpha|=\ell\)) which satisfy:
	\begin{enumerate}[resume, label=(C\arabic*)]
		\item \(T(x)\) maps \(M\to V\) and is homogeneous of degree \(k-n-1\),
		\item \(P(y)\) maps \(F \to M\) and is homogeneous of degree \(\ell+1\), 
		\item and either:
		\begin{enumerate}
			\item  as linear mappings \(F\to V\),
			\begin{align*}
				\sum_{|\alpha|=\ell}\partial^{\alpha}_{y}(T(x)P(y))L_{\alpha}&=\sum_{|\alpha|=\ell}\sum_{e_{i}\leq \alpha}{\alpha \choose e_i}\partial_{x_i}G(x)\partial^{\alpha-e_i}_yK(y)L_{\alpha} \tag{Weak C6} \label{WTP}
			\end{align*}
			\item for every \(|\alpha|=\ell\), as linear mappings \(F\to V\),
			\begin{align*}
				\partial^{\alpha}_{y}(T(x)P(y))&=\sum_{e_{i}\leq \alpha}{\alpha \choose e_i}\partial_{x_i}G(x)\partial^{\alpha-e_i}_yK(y) \tag{C6} \label{TP}
			\end{align*}
			
		\end{enumerate}
		where \(K(y)=\sum_{|\alpha|=\ell}K_{\alpha}y^{\alpha}\) in both of the above.
	\end{enumerate} 
	Then it can be shown
	\begin{align*}
		\left|\int_{\R^n}y\cdot DG(x)A(D)u(y)\rho\left(\frac{|y|}{|x|}\right)dy\right|&\lesssim \sum_{j=1}^{\ell}\int_{\R^n}|x|^{k-n}|A(D)u(y)|\left|\rho^{(j)}\left(\frac{|y|}{|x|}\right)\right|dy
	\end{align*}
	with implied constant independent of \(u\).
\end{prop}
Note that the existence of \ref{G}, \ref{L}, \ref{K} follows from \(A(D)\) being elliptic and canceling. However we know very few examples of differential operators which satisfy \eqref{TP}, which suggests that more work is needed in this direction, and we leave open the problem of if the elliptic and canceling assumptions imply \eqref{WTP} (or something similar). We do know though, that not every representation of \(A(D)\) admits a consistent system, see Remark \ref{bad3dgrad}. On the other hand we have the following nice example:
\begin{eg}
	\(\nabla\) on \(\R^2\):\\
	Choose 
	\begin{align*}
		L(D)&=\begin{bmatrix} \partial_{y_2} & -\partial_{y_1}	\end{bmatrix} \\
		K(y)&=\begin{bmatrix}
			y_2\\
			-y_1
		\end{bmatrix}\\
		G(x)&=\begin{bmatrix}
			{x_1} & {x_2}
		\end{bmatrix}\frac{1}{|x|^2}
	\end{align*}
	and
	\begin{align*}
		T(x)&=
		\begin{bmatrix}
			x_2^2-x_1^2 & 2x_1x_2
		\end{bmatrix}\frac{1}{|x|^4}\\
		P(y)&=\begin{bmatrix}
			y_1y_2\\
			\frac{y_2^2-y_1^2}{2}
		\end{bmatrix}.
	\end{align*}
	One may check
	\begin{align*}
		\partial_{x_1}G(x)K(y)&=\partial_{y_1}T(x)P(y)\\
		\partial_{x_2}G(x)K(y)&=\partial_{y_2}T(x)P(y)
	\end{align*}
	which would imply \eqref{TP}. 
\end{eg}
\begin{remark}\label{bad3dgrad}\(\)\\
	Here we demonstrate that some obvious choices involving \(\nabla\) on \(\R^3\) yield an inconsistent system. Choose:
	\begin{align*}
		G(x)&=\begin{bmatrix}
			\partial_{x_1} & \partial_{x_2} & \partial_{x_3}
		\end{bmatrix}\frac{1}{4\pi|x|}.
	\end{align*}
	The representation
	\begin{align*}
		L(D)&=\frac{1}{2}\begin{bmatrix}
			0 & -\partial_{y_3} & \partial_{y_2}\\
			\partial_{y_3} & 0 & -\partial_{y_1}\\
			-\partial_{y_2} & \partial_{y_1} & 0
		\end{bmatrix}\\
		K(y)&=y_1\begin{bmatrix}
			0 & 0 & 0\\
			0 & 0 & 1\\
			0 & -1 & 0
		\end{bmatrix}+y_2\begin{bmatrix}
			0 & 0 & -1\\
			0 & 0 & 0\\
			1 & 0 & 0
		\end{bmatrix}+y_3\begin{bmatrix}
			0 & 1 & 0\\
			-1 & 0 & 0\\
			0 & 0 & 0
		\end{bmatrix}
	\end{align*}
	does not admit a solution to the system
	\begin{align*}
		\partial^{\alpha}_yK_{\alpha,i}(y)&=y_iK_{\alpha}\\
		\sum_{i=1}^3\partial_{x_i}G(x)K_{\alpha,i}(y)&=\sum_{i=1}^3\partial_{x_i}G(x)K_{\alpha',i}(y)
	\end{align*}
	even if it is underdetermined. Note that
	\begin{align*}
		\partial^{\alpha}_y\left(\sum_{i=1}^3\partial_{x_i}G(x)K_{\alpha,i}(y)\right)&=\sum_{i=1}^3\partial_{x_i}G(x)y_iK_{\alpha}
	\end{align*}
	and hence
	\begin{align*}
		\partial_{y}^{e_1}\partial^{\alpha}_y\left(\sum_{i=1}^3\partial_{x_i}G(x)K_{\alpha,i}(y)\right)&=\partial_{x_1}G(x)K_{\alpha}.
	\end{align*}
	If we choose \(\alpha=e_2\), and \(\alpha'=e_1\), then the second condition gives
	\begin{align*}
		\sum_{i=1}^3\partial_{x_i}G(x)K_{e_2,i}(y)&=\sum_{i=1}^3\partial_{x_i}G(x)K_{e_1,i}(y)
	\end{align*}
	but if we apply \(\partial_{y}^{e_1+e_2}\) to both sides, by the previous calculation one has
	\begin{align*}
		\partial_{x_1}G(x)K_{e_2}&=\partial_{x_2}G(x)K_{e_1}
	\end{align*}
	which requires that
	\begin{align*}
		\begin{bmatrix}
			\partial_{x_1}\partial_{x_3} & 0 & -\partial_{x_1}^2
		\end{bmatrix}|x|^{-1}&=\begin{bmatrix}
			0 & -\partial_{x_2}\partial_{x_3} & \partial_{x_2}^2
		\end{bmatrix}|x|^{-1}
	\end{align*}
	which is clearly false. \\
\end{remark}

In contrast to the previous non-example, we leave open the following question: 
\begin{question} \(\)\\
	Does the following representation of \(\nabla\) on \(\R^3\) admit property \eqref{WTP}? 
\end{question}
We have in mind:
\begin{align*}
	L(D)&=\begin{bmatrix}
		\partial_{y_2}\partial_{y_3} & \partial_{y_1}\partial_{y_3} &-2\partial_{y_1}\partial_{y_2}
	\end{bmatrix}\\
	K(y)&=y_1y_2\begin{bmatrix}
		0\\
		0\\
		-\frac{1}{2}
	\end{bmatrix}+y_1y_3\begin{bmatrix}
		0\\
		1\\
		0
	\end{bmatrix}+y_2y_3\begin{bmatrix}
		1\\
		0\\
		0
	\end{bmatrix}.
\end{align*}
Then construct
\begin{align*}
	K_1(y)&=\begin{bmatrix}
		y_1y_2y_3\\
		\frac{y_1^2y_3}{2}\\ 
		-\frac{y_1^2y_2}{4}
	\end{bmatrix}\\
	K_2(y)&=\begin{bmatrix}
		\frac{y_2^2y_3}{2}\\
		y_1y_2y_3\\
		-\frac{y_1y_2^2}{4}
	\end{bmatrix}\\
	K_3(y)&=\begin{bmatrix}
		\frac{y_2y_3^2}{2}\\
		\frac{y_1y_3^2}{2}\\
		-\frac{y_1y_2y_3}{2}
	\end{bmatrix}.
\end{align*}
One has that \(\partial_{y_i}K_i(y)=K(y)\). We set
\begin{align*}
	T(x)P(y)&=\partial_{x_1}G(x)K_1(y)+\partial_{x_2}G(x)K_2(y)+\partial_{x_3}G(x)K_3(y).
\end{align*} 
Can we use
\begin{align*}
	\partial_{x_3}G(x)&=-\partial_{x_1}G(x)-\partial_{x_2}G(x)
\end{align*}
for all \(x\in\R^3\setminus\{0 \}\) in some way to show \eqref{WTP}?
\section*{Acknowledgements}
The author would like to express his heartfelt gratitude to Po-Lam Yung for his patience and support throughout this work. The author is also deeply indebted to Jan Kristensen for hosting him as an exchange student while some of this work was completed, and would like to thank the University of Oxford Mathematical Institute for their hospitality . Finally, the author wishes to acknowledge Carlos P\'{e}rez Moreno's encouragement. 

The author is supported by a Australian Government Research Training Program Scholarship and by Australian Research Council Grant FT20010039. They also received support from the AIM research community on Fourier restriction and the ANU Vice-Chancellor's HDR Travel Grant.
\section*{Data Availability}
Data sharing is not applicable, since no data was generated or analysed for this research.

\end{document}